\documentclass[11pt]{amsart}
\usepackage{color}

\usepackage{amsmath,amssymb,amsthm,amscd,amsrefs,mathtools}
\usepackage{graphicx,subfigure,hyperref}

\makeatletter
\newcommand{\newreptheorem}[2]{\newtheorem*{rep@#1}{\rep@title}\newenvironment{rep#1}[1]{\def\rep@title{#2 \ref{##1}}\begin{rep@#1}}{\end{rep@#1}}}
\makeatother

\newtheorem{thm}{Theorem}[section]
\newtheorem*{thm*}{Theorem}
\newreptheorem{thm}{Theorem}
\newtheorem{lem}[thm]{Lemma}
\newtheorem*{lem*}{Lemma}
\newtheorem{prop}[thm]{Proposition}
\newtheorem*{prop*}{Proposition}

\theoremstyle{definition}

\numberwithin{equation}{section}

\newcommand{\Ad}[1]{\text{Ad}_{#1}}

\newcommand{\SL}{\operatorname{SL}}
\newcommand{\GL}{\operatorname{GL}}
\newcommand{\Hom}{\operatorname{Hom}}

\newcommand{\sslash}{\mathbin{/\mkern-6mu/}}

\makeatletter
\@namedef{subjclassname@2020}{%
  \textup{2020} Mathematics Subject Classification}
\makeatother

\title[Deformations of reducible $\SL(n,\mathbb{C})$ representations]{
Deformations of reducible $\SL(n,\mathbb{C})$  representations of fibered 3-manifold groups}

\author{Kenji Kozai}
\address{Department of Mathematics,
Southern Connecticut State University,
501 Crescent Street, New Haven, CT 06515,
USA}
\email{kozaik1@southernct.edu}


\begin{document}

\begin{abstract}
Let $M_\phi$ be a surface bundle over a circle with monodromy $\phi:S \rightarrow S$.
We study deformations of certain reducible representations of $\pi_1(M_\phi)$
into $\SL(n,\mathbb{C})$, obtained by composing a reducible representation into
$\SL(2,\mathbb{C})$ with the irreducible representation $\SL(2,\mathbb{C})
\rightarrow \SL(n,\mathbb{C})$. In particular, we show that under certain
conditions on the eigenvalues of $\phi^*$,
the reducible representation is contained in a $(n+1+k)(n-1)$ dimensional component
of the representation variety, where $k$ is the number of components of
$\partial M_\phi$. This result applies to mapping tori of pseudo-Anosov
maps with orientable invariant foliations whenever 1 is not an eigenvalue of the
induced map on homology, where the reducible representation is also a limit of
irreducible representations.
\end{abstract}

\maketitle

\section{Introduction}

Suppose that $S=S_{g,p}$ is a surface of genus $g$ with $p\geq 1$ punctures, where
$2g+p>2$. Then $S$ admits a hyperbolic structure. If $\phi:S \rightarrow S$ is a
homeomorphism, we can form the mapping torus $M_\phi = S \times [0,1] /
(x,1) \sim (\phi(x),0)$. Whenever $\lambda^2$ is an eigenvalue of $\phi^*:
H^1(S) \rightarrow H^1(S)$ with eigenvector $(a_1,\dots,a_{2g+p-1})^T$ with
respect to a generating set $\{[\gamma_1],\dots,[\gamma_{2g+p-1}]\}$ of
$H^1(S)$, we obtain a reducible representation $\rho_\lambda:\pi_1(M_\phi)
\rightarrow \SL(2,\mathbb{C})$ by defining,
\begin{align*}
	\rho_\lambda(\gamma_i) &= \begin{pmatrix} 1 & a_i\\0&1\end{pmatrix}\\
	\rho_\lambda(\tau) &=\begin{pmatrix} \lambda & 0\\0&\lambda^{-1} \end{pmatrix},
\end{align*}
where $\tau$ is the generator of the fundamental group of the $S^1$ base of the fiber bundle
$S \rightarrow M_\phi \rightarrow S^1$. (Recall that a representation $\rho:G \rightarrow
\GL(n,\mathbb{C})$ is \textit{reducible} if the image $\rho(G)$ preserves a proper
subspace of $\mathbb{C}^n$, and otherwise is called \textit{irreducible}.)

When $M_\phi$ is the complement of a knot $K$ in $S^3$, this
observation was originally made by Burde \cite{burde67} and de Rham
\cite{derham68}. Furthermore, the Alexander polynomial is the characteristic
polynomial of $\phi^*$, so the condition on $\lambda$ is equivalent to the
condition that $\lambda^2$ is a root of the Alexander polynomial
$\Delta_K(t)$. It was shown in \cite{heusener01} that the non-abelian, metabelian,
reducible representation $\rho_\lambda$ is
the limit of irreducible representations if $\lambda^2$ is a simple root of
$\Delta_K(t)$. Heusener and Medjerab \cite{heusener16} have also shown using
an inductive argument that
the conclusion still holds in $\SL(n,\mathbb{C})$, $n \geq 3$, if $\rho_\lambda$ is
composed with the irreducible representation $r_n:\SL(2,\mathbb{C})\rightarrow
\SL(n,\mathbb{C})$. These results apply even if the knot complement is not fibered,
as long as $\lambda^2$ is a simple root of $\Delta_K(t)$.

In this paper, we show that reducible $\SL(n,\mathbb{C})$ representations of
fibered 3-manifolds groups obtained as the composition $\rho_{\lambda,n} = r_n
\circ \rho_\lambda$ can be deformed to irreducible representations using a more
direct calculation of the deformation space using coordinates for
$\mathfrak{sl}(n,\mathbb{C})$. If
the punctures form a single orbit under $\phi$ and the mapping torus $M_\phi$ is the
complement of a fibered knot, then the results of \cite{heusener01} and
\cite{heusener16} apply. The main result in Theorem \ref{thm:smooth}
also covers the cases where $M_\phi$ is the complement of a fibered
link $L$ with $k\geq 2$ components $L_1,\dots, L_k$, or a
$k$-cusped fibered manifold which is not a link
complement. In the statement of Theorem \ref{thm:smooth}, $\bar{\phi}$ is the
homeomorphism on $\bar{S} = S_{g,0}$ obtained from $\phi$ by filling in the
$p$ punctures of $S_{g,p}$.
This defines a homeomorphism $\bar{\phi}: \bar{S}\rightarrow \bar{S}$.

\begin{thm}\label{thm:smooth}
	Suppose that $\lambda^2$ is a simple eigenvalue of $\phi^*$. If $|\lambda | \neq 1$,
	$\bar{\phi}^*:H^1(\bar{S}) \rightarrow H^1(\bar{S})$ does
	not have 1 as an eigenvalue, and if for each $2 \leq j \leq n-1$, we have that
	$\lambda^{2j}$ is not an eigenvalue of $\phi^*$, then $\rho_{\lambda,n}$ is a
	smooth point of the representation variety $R(\pi_1(M_\phi),\SL(n,\mathbb{C}))$,
	contained in a unique component of dimension $(n+1+k)(n-1)$.
\end{thm}

Note that for a knot complement, the Alexander polynomial satisfies $\Delta_K(1)
= \pm 1$. Hence for a knot complement, the condition that
$\bar{\phi}^*:H^1(\bar{S}) \rightarrow H^1(\bar{S})$ does not have 1 as an eigenvalue
(in the fibered case) or the corresponding condition that $1$ is not a root of $\Delta_K(t)$
(in the non-fibered case)
is automatically satisfied. For a generic mapping torus, a fixed point of $\bar{\phi}^*$
implies that the closed manifold obtained as the mapping torus of $\bar{\phi}$ has
second Betti number at least 2, in which case the manifold fibers over a circle
in infinitely many ways \cite{thurston86}. Heuristically, this leads to more
infinitesimal deformations. When the local dimension of infinitesimal
dimensions is higher than half the dimension of $H^1(\partial M_\phi)$, the standard
techniques using Poincar\'{e} duality to show smoothness of the space of
representations cannot be used. Whether the reducible representation can be
obtained as a limit of irreducible representations in this case is unknown.

When $\phi$ is a
pseudo-Anosov element of the mapping class group, $\lambda^2$ is the dilatation
factor of $\phi$, and the $p$ punctures are exactly the singular points of the
invariant foliations of $\phi$, $\rho_\lambda=\rho_{\lambda,2}$ is shown to have
deformations to irreducible representations under some additional conditions on
the eigenvalues of $\bar{\phi}^*$, the map on the closed surface $S_g$, in
\cite{kozai16}. We show that under the same hypotheses, the same holds for
$\rho_{\lambda,n}$ when $n >2$.

\begin{thm}\label{thm:pseudoanosov}
	Suppose that $\lambda^2$ is the dilatation of a pseudo-Anosov map $\phi$ such
	that the stable and unstable foliations are orientable, and the singular points
	coincide with the punctures of $S$. Suppose also that
	$1$ is not an eigenvalue of $\bar{\phi}^*$. Then $\rho_{\lambda,n}$ is a limit of
	irreducible $\SL(n,\mathbb{C})$ representations  and is
	a smooth point of $R(\pi_1(M_\phi),\SL(n,\mathbb{C}))$, contained in a unique
	component of dimension $(n+1+k)(n-1)$.
\end{thm}

In Section \ref{sec:slnc}, we give the basic definitions and background about
representations of $\SL(2,\mathbb{C})$ into $\SL(n,\mathbb{C})$. Section
\ref{sec:deformations} discusses the general theory of deformations,
and Section \ref{sec:results} contains the main results, including relevant
cohomological calculations.

\section{Representations into $\SL(n,\mathbb{C})$}\label{sec:slnc}

For notational convenience, we denote $\SL(n) = \SL(n,\mathbb{C})$, $\mathfrak{sl}(n)
= \mathfrak{sl}(n,\mathbb{C})$, $\GL(n) = \GL(n,\mathbb{C})$, and $\Gamma_\phi=
\pi_1(M_\phi)$. 
Note that we have the following identities in $\SL(2)$:
\begin{align}
	\begin{pmatrix} 1 & a \\ 0 & 1 \end{pmatrix}
	\begin{pmatrix} 1 & b \\ 0 & 1 \end{pmatrix}
	&= \begin{pmatrix} 1 & a+b \\ 0 & 1 \end{pmatrix}, \label{eqn:translation}\\
	\begin{pmatrix} \lambda & 0 \\ 0 & \lambda^{-1} \end{pmatrix}
	\begin{pmatrix} 1 & a \\ 0 & 1 \end{pmatrix}
	\begin{pmatrix} \lambda & 0 \\ 0 & \lambda^{-1} \end{pmatrix}^{-1}
	&= \begin{pmatrix} 1 & \lambda^2 a \\ 0 & 1 \end{pmatrix}.\notag
\end{align}
Thus, if $\lambda^2$ is an eigenvalue of $\phi^*:H^1(S) \rightarrow H^1(S)$,
$\{[\gamma_1],\dots,[\gamma_{2g+p-1}]\}$ generate $H^1(S)$, and
$(a_1,\dots,a_{2g+p-1})^T$ is an eigenvector for $\lambda^2$, we can define
\begin{align*}
	\rho_\lambda(\gamma_i) &= \begin{pmatrix} 1 & a_i\\0&1\end{pmatrix}\\
	\rho_\lambda(\tau) &=\begin{pmatrix} \lambda & 0\\0&\lambda^{-1} \end{pmatrix}.
\end{align*}
Since $\pi_1(\Gamma_\phi)$ is a semi-direct product of the free group $\pi_1(S)=
\langle \gamma_1,\dots,\gamma_{2g+p-1}\rangle$ with $\pi_1(S^1)=\langle\tau\rangle$
satisfying the relations $\tau\gamma_i\tau^{-1}=\phi(\gamma_i)$ and $\phi^*$ maps the
vector $(a_1,\dots,a_{2g+p-1})^T$ to $ \lambda^2 (a_1,\dots,a_{2g+p-1})^T$, the identities 
\eqref{eqn:translation} imply that this defines a representation
$\rho_\lambda: \Gamma_\phi \rightarrow \SL(2)$.

We now describe representations of $\SL(2)$ into $\SL(n)$, which we will compose
with $\rho_\gamma$ to obtain representations $\Gamma_\phi \rightarrow \SL(n)$.
A more general version of the discussion in this section can be found in
\cite[Section 4]{heusener16}.

Let $R = \mathbb{C}[X,Y]$ be the polynomial algebra on two variables. We have an
action of $\SL(2)$ on $R$ by,
\begin{align*}
	\begin{pmatrix} a & b\\c&d\end{pmatrix} \cdot X &= dX - bY\\
	\begin{pmatrix} a& b\\c&d\end{pmatrix} \cdot Y & = -cX + aY,
\end{align*}
for $\begin{psmallmatrix}a & b\\c&d\end{psmallmatrix} \in \SL(2)$. Let
$R_{n-1} \subset R$ denote the $n$-dimensional subspace of homogenous
polynomials of degree $n-1$, generated by $X^{\ell-1}Y^{n-\ell},
1 \leq \ell \leq n$.
The action of $\SL(2)$ leaves $R_{n-1}$ invariant, turning $R_{n-1}$ into a
$\SL(2)$ module, and we obtain a representation $r_n: \SL(2) \rightarrow
\GL(R_{n-1})$. We can identify $R_{n-1}$ with $\mathbb{C}^n$ by identifying
the basis elements $\{X^{\ell-1}Y^{n-\ell}\}$ with the standard basis elements
$\{e_\ell\}$ of $\mathbb{C}^n$. The induced isomorphism turns $r_n$ into a
representation $\SL(2) \rightarrow \GL(n) \cong \GL(R_{n-1})$,
which we will also call $r_n$. The representation $r_n$ is \textit{rational},
that is the coefficients of the matrix coordinates of
$r_n\begin{psmallmatrix}a&b\\c&d\end{psmallmatrix}$ are polynomials in
$a,b,c,d$.

We have the following two well-known results about $r_n$.

\begin{lem}\cite[Lemma 3.1.3(ii)]{springer77}\label{lem:rn_irreducible}
	The representation $r_n$ is irreducible.
\end{lem}

\begin{lem}\cite[Lemma 3.2.1]{springer77}\label{lem:rn_unique}
	Any irreducible rational representation of $\SL(2,\mathbb{C})$ is conjugate to some
	$r_n$.
\end{lem}

It is easy to check that $r_n$ maps the unipotent matrices
$\begin{psmallmatrix}1&b\\0&1\end{psmallmatrix}$
and $\begin{psmallmatrix}1&0\\c&1\end{psmallmatrix}$ to unipotent elements of
$\SL(R_{n-1})$, and the diagonal element
$\begin{psmallmatrix}a & 0 \\ 0 & a^{-1}\end{psmallmatrix}$ is mapped to the
diagonal element $\text{diag}(a^{n-1},a^{n-3},\dots,a^{-n+1})$. Since these
elements generate $\SL(2)$, the image of $r_n$ lies in
$\SL(R_{n-1}) \cong \SL(n)$.

We now define $\rho_{\lambda,n} = r_n \circ \rho_\lambda$. As we will only be considering
the case when $\lambda^2$ is a simple eigenvalue of $\phi^*$ and the above lemmas
imply the uniqueness of $r_n$, this gives a well-defined and unique (up to conjugation)
representation $\rho_{\lambda,n}:\Gamma_\phi \rightarrow \SL(n)$.

By composing $\rho_{\lambda,n}$
with the adjoint representation, we also obtain an action of $\Gamma_\phi$ on 
$\mathfrak{sl}(n)$, turning it into a $\Gamma_\phi$ module. The following
decomposition is a consequence of the Clebsch-Gordan formula (see,
for example, \cite[Lemma 1.4]{menalferrer12-1}).

\begin{lem}\label{lem:decomposition}
	With the $\Gamma_\phi$ module structure,
	$\mathfrak{sl}(n) \cong \oplus_{j=1}^{n-1} R_{2j}$.
\end{lem}

This decomposition will be used to calculate the infinitesimal deformations
of $\rho_{\lambda,n}$.

\section{Infinitesimal deformations}\label{sec:deformations}

In this section, let $M$ be a 3-manifold with finitely many torus boundary
components $\partial M =  \sqcup_{i=1}^k T_i$ and $\Gamma=\pi_1(M)$.
For each boundary torus $T_i$, the inclusion map $\iota: T_i \rightarrow M$
induces a map from $\pi_1(T_i)$ to a conjugacy class of subgroups isomorphic
to $\pi_1(T_i) \cong \mathbb{Z} \times \mathbb{Z}$ in $\pi_1(M)$. To each boundary
component $T_i$, we associate $\pi_1(T_i)$ with a representative subgroup $\Delta_i$
in $\Gamma$.
Let $R(\Gamma,\SL(n))=\Hom(\Gamma,\SL(n))$ be the variety of
representations of $\Gamma$ into $\SL(n)$ and $X(\Gamma,\SL(n)) =
R(\Gamma,\SL(n))\sslash \SL(n)$ be the $\SL(n)$ character variety, where the
quotient is the GIT quotient as $\SL(n)$ acts by conjugation.

Suppose $\rho:\Gamma\rightarrow \SL(n)$ is a representation.
The group of twisted cocycles  $Z^1(\Gamma;\mathfrak{sl}(n)_\rho)$
is defined as the set of maps $z:\Gamma \rightarrow \mathfrak{sl}(n)$ that satisfy the
twisted cocycle condition
\begin{equation}
	z(ab) = z(a) + \Ad{\rho(a)} z(b),\label{eqn:cocycle}
\end{equation}
which can be interpreted as the derivative of the homomorphism condition for
a smooth family of representation $\rho_t$ at $\rho$.
The derivative of the triviality condition that $\rho_t$ is a smooth family of
representations obtained by conjugating $\rho$ gives the coboundary condition,
\begin{equation}
	z(\gamma) = u - \Ad{\rho(\gamma)}u, \label{eqn:coboundary}
\end{equation}
and $B^1(\Gamma;\mathfrak{sl}(n)_\rho)$ is defined as the
set of coboundaries, or the cocycles
satisfying Equation \eqref{eqn:coboundary}. The quotient is defined to be
\begin{equation*}
	H^1(\Gamma;\mathfrak{sl}(n)_\rho) = Z^1(\Gamma;\mathfrak{sl}(n)_\rho) /
		B^1(\Gamma;\mathfrak{sl}(n)_\rho).
\end{equation*}
Weil \cite{weil64, lubotzky85} has  noted that $Z^1(\Gamma;\mathfrak{sl}(n)_\rho)$
contains the tangent space to $R(\Gamma,\SL(n))$ at $\rho$ as a subspace.
The following tools can be used to determine if the representation variety
is smooth at $\rho$ so that we can study the space of cocycles to
determine the first order behavior of deformations of a representation $\rho$.
In the following proposition, $C^1(\Gamma;\mathfrak{sl}(n)_\rho)$ denotes the
set of cochains $\{c:\Gamma \rightarrow \mathfrak{sl}(n)\}$.

\begin{prop}[\cite{heusener16}, Lemma 3.2; \cite{heusener01}, Proposition 3.1]
	Let $\rho \in R(\Gamma,\SL(n))$, $u_i \in C^1(\Gamma;
	\mathfrak{sl}(n)_\rho)$, $ 1 \leq i \leq j$ be given, and $\mathbb{C}[[t]]$
	denote the set of formal power series in $t$ with coefficients in
	$\mathbb{C}$. If
	\begin{equation*}
		\rho^j(\gamma) = \exp(\sum_{i=1}^j t^i u_i(\gamma))\rho(\gamma)
	\end{equation*}
	is a homomorphism into $\SL(n,\mathbb{C}[[t]])$ modulo
	$t^{j+1}$, then there exists an obstruction class $\zeta_{j+1}^{(u_1,\dots,u_k)} \in
	H^2(\Gamma;\mathfrak{sl}(n)_\rho)$ such that:
	\begin{enumerate}
		\item There is a cochain $u_{j+1}:\Gamma \rightarrow
			\mathfrak{sl}(n)$ such that
			\begin{equation*}
				\rho^{j+1}(\gamma) = \exp(\sum_{i=1}^{j+1} t^i u_i(\gamma)) 
					\rho(\gamma)
			\end{equation*}
			is a homomorphism modulo $t^{j+2}$ if and only if $\zeta_{j+1}=0$.
		\item The obstruction $\zeta_{j+1}$ is natural, i.e. if $f$ is a homomorphism
			then $f^* \rho^j := \rho^j \circ f$ is also a homomorphism modulo $t^{j+1}$
			and $f^*(\zeta_{j+1}^{(u_1,\dots,u_j)}) = \zeta_{j+1}^{(f^* u_1,\dots,f^* u_j)}$.
	\end{enumerate}
\end{prop}

We will apply the previous proposition to the restriction map $\iota^*$ on cohomology, which
is induced by the inclusion map $\iota:\partial M \rightarrow M$. As
$\partial M$ consists of a disjoint union of tori, we will need to understand
$H^1(\Delta_i;\mathfrak{sl}(n)_{r_n\circ\rho})$. Recall that a \textit{hyperbolic} element
of $\SL(2)$ is an element that acts on $\mathbb{H}^3$ with no fixed points in
$\mathbb{H}^3$ and two fixed points on $\partial \mathbb{H}^3$. Such elements are
characterized by being conjugate in $\SL(2)$ to a diagonal matrix with distinct eigenvalues
that are not on the unit circle.

\begin{lem}\label{lem:torus_dimension}
	Suppose $\rho:\mathbb{Z} \times \mathbb{Z} \rightarrow \SL(2)$ contains a hyperbolic
	element in its image. Then
	$\dim H^1(\mathbb{Z} \times \mathbb{Z};\mathfrak{sl}(n)_{r_n\circ \rho}) = 2(n-1)$.
\end{lem}

\begin{proof}
	Suppose $\gamma\in \mathbb{Z} \times \mathbb{Z}$ such that $\rho(\gamma)$ is a hyperbolic element
	in $\SL(2)$. Then, up to conjugation,
	\begin{equation*}
		\rho(\gamma) = \begin{pmatrix} a & 0 \\ 0 & a^{-1}\end{pmatrix},
	\end{equation*}
	for some $|a|>1$.
	The image of such an element under the irreducible representation $r_n:\SL(2)
	\rightarrow \SL(n)$ is conjugate to a diagonal matrix
	with $n$ distinct eigenvalues. Hence any nearby representation $\rho':
	\mathbb{Z} \times \mathbb{Z}\rightarrow \SL(n)$  is conjugate to a diagonal matrix with distinct
	entries. In other words, up to coboundary, we can assume that any class
	$[z] \in H^1(\mathbb{Z} \times \mathbb{Z};\mathfrak{sl}(n)_{r_n\circ\rho})$ has the form of a diagonal matrix
	$z(\gamma) = \text{diag}(y_1,y_2,\dots,y_n)$ where $\text{tr}~z(\gamma) = 0$.
	Since for any other $\gamma' \in \mathbb{Z} \times \mathbb{Z}$,
	we have that $\gamma'$ commutes with $\gamma$, $z(\gamma')$ must also be diagonal,
	so the dimension of $H^1(\mathbb{Z} \times \mathbb{Z};\mathfrak{sl}(n)_{r_n\circ\rho})$ is
	$2(n-1)$.
\end{proof}

\begin{lem}\label{lem:injectivity}
	Let $\rho:\pi_1(M) \rightarrow \SL(2)$ be a non-abelian
	representation such that $\rho(\Delta_i)$ contains a hyperbolic element
	for each subgroup $\Delta_i$ of $\pi_1(M)$
	associated to a boundary component
	$T_i$ of $\partial M$. If $\dim H^1(
	\Gamma;\mathfrak{sl}(2)_{r_n\circ\rho}) = k(n-1)$ where $k$ is the number
	of components of $\partial M$, then $\iota^*:H^2(M;\mathfrak{sl}(n)_{
	r_n\circ\rho}) \rightarrow H^2(\partial M;\mathfrak{sl}(n)_{r_n\circ\rho})$
	is injective.
\end{lem}

\begin{proof}
	We have the cohomology exact sequence for the pair $(M,\partial M)$
	\begin{equation*}
		\begin{CD}
			H^1(M,\partial M)@>>>
				H^1(M) @>\alpha>>
				H^1(\partial M) \\
				@>\beta>>
				H^2(M,\partial M) @>>>
				H^2(M) \\
				@>\iota^*>>
				H^2(\partial M) @>>>
				H^3(M,\partial M)@>>> \cdots
		\end{CD}
	\end{equation*}
	where all cohomology groups are taken to be with the twisted coefficients
	$\mathfrak{sl}(n)_{r_n\circ\rho}$.
	A standard Poincar\'{e} duality argument \cite{heusener01,hodgson98,porti97}
	implies that $\alpha$ has half-dimensional image in $H^1(\partial M)$.
	By Lemma \ref{lem:torus_dimension},
	\begin{equation*}
		\dim H^1(\Delta_i) = 2(n-1),
	\end{equation*}
	as long as $\rho(\Delta_i)$ contains a hyperbolic element for each $i$.
	We can identify $H^1(\partial M) \cong \oplus_{i=1}^k H^1(\Delta_i)$, which
	has dimension $2k(n-1)$. Since $H^1(M) \cong H^1(\Gamma)$ has
	dimension $k(n-1)$, then $\alpha$ is injective. Since $\beta$ is
	dual to $\alpha$ under Poincar\'{e} duality, then $\beta$ is surjective. This
	implies that $\iota^*$ is injective.
\end{proof}

We now utilize the previous facts to determine sufficient conditions for deforming
representations.

\begin{prop}\label{prop:dim_smooth}
	Let $\rho:\Gamma \rightarrow \SL(2)$ be a non-abelian
	representation such that
	$\rho(\Delta_i)$ contains a hyperbolic element for each subgroup $\Delta_i$.
	If $H^1(\Gamma;\mathfrak{sl}(2)_{r_n\circ\rho}) = k(n-1)$ where $k$ is the number
	of components of $\partial M$, then $r_n\circ\rho$ is a smooth point of the representation
	variety $R(\Gamma,\SL(n))$, and it is contained in a unique component
	of dimension $(n+1+k)(n-1) - \dim
	H^0 (\Gamma; \mathfrak{sl}(n)_{r_n\circ\rho})$.
\end{prop}

\begin{proof}
	We begin by showing that every cocyle in $Z^1(\Gamma;
	\mathfrak{sl}(n)_{r_n\circ\rho})$ is integrable.
	
	Suppose we have $u_1,\dots,u_j : \Gamma \rightarrow
	\mathfrak{sl}(n)$ such that 
	\begin{equation*}
		\rho_n^j(\gamma) = \exp(\sum_{i=1}^j t^i u_i(\gamma))\rho(\gamma)
	\end{equation*}
	is a homomorphism modulo $t^{j+1}$. By Lemma \ref{lem:torus_dimension}
	and \cite{richardson79}, the restriction
	of $\rho_n$ to $\Delta_i$ is a smooth point of the representation variety
	$R(\Delta_i,\SL(n))$. Hence $\rho_n^j|_{\pi_1(T_i)}$ extends
	to a formal deformation of
	order $j+1$ by the formal implicit function theorem (see \cite{heusener01},
	Lemma 3.7). This implies that the restriction of $\zeta_{j+1}^{(u_1,\dots,u_j)}$
	to each component $H^2(T_i) < H^2(\partial M)$ vanishes.
	
	As $H^2(\partial M) = \oplus_{i=1}^k H^2(T_i)$,
	hence, $\iota^* \zeta_{j+1}^{(u_1,\dots,u_j)} = \zeta_{j+1}^{(\iota^*u_1,\dots,\iota^*u_j)} = 0$.
	The injectivity of $\iota^*$ follows from Lemma \ref{lem:injectivity}
	and implies that $\zeta_{j+1}^{(u_1,\dots,u_j)} = 0$. Hence,
	the homomorphism can be extended to a deformation
	$(r_n\circ\rho)^{j+1}$ of order $j+1$,  and inductively to a formal deformation
	$(r_n\circ\rho)^\infty$.
	
	Applying \cite[Proposition 3.6]{heusener01} to the formal deformation
	$(r_n\circ\rho)^\infty$ results in a convergent deformation. Hence, $r_n\circ\rho$ is a smooth
	point of the representation variety. 
	
	As in \cite{heusener16}, we note that the exactness of
	\begin{equation*}
		1 \rightarrow H^0(\Gamma;\mathfrak{sl}(n)_{r_n\circ\rho})
			\rightarrow \mathfrak{sl}(n)_{r_n\circ\rho}
			\rightarrow B^1(\Gamma;\mathfrak{sl}(n)_{r_n\circ\rho})
	\end{equation*}
	implies that
	\begin{equation*}
		\dim B^1(\Gamma;\mathfrak{sl}(n)_{r_n\circ\rho}) =
		n^2-1 - \dim H^0(\Gamma;\mathfrak{sl}(n)_{r_n\circ\rho}).
	\end{equation*}
	Thus, we conclude that the local dimension of $R(\Gamma,\SL(n))$
	is
	\begin{equation*}
		\dim Z^1(\Gamma;\mathfrak{sl}(n)_{r_n\circ\rho}) = 
			(n+1+k)(n-1)-\dim H^0(\Gamma;\mathfrak{sl}(n)_{r_n\circ\rho}).
	\end{equation*}
	That it is in a unique component follows from \cite[Lemma 2.6]{heusener01}.
\end{proof}

\section{Deforming $\rho_{\lambda,n}$}\label{sec:results}

We will now show that $\rho_{\lambda,n}$ satisfies the conditions in Proposition
\ref{prop:dim_smooth} so that $\rho_{\lambda,n}$ can be deformed within a
neighborhood of representations. This will entail a computation of the
dimension of the cohomology group $H^1(\Gamma_\phi;
\mathfrak{sl}(n)_{\rho_{\lambda,n}})$. By the decomposition in Lemma
\ref{lem:decomposition}, the cohomology group $H^1(
\Gamma_\phi;\mathfrak{sl}(n)_{\rho_{\lambda,n}})$ is a direct sum
\begin{equation*}
	H^1(\Gamma_\phi;\mathfrak{sl}(n)_{\rho_{\lambda,n}}) \cong
		\oplus_{j=1}^{n-1} H^1(\Gamma_\phi;R_{2j}),
\end{equation*}
so it suffices to compute the dimensions of $H^1(\Gamma_\phi;R_{2j})$, for
$1 \leq j \leq n-1$.

To simplify the computations which follow, we give a presentation of
$\Gamma_\phi$ with an additional generator $\gamma_{2g+p}$. We will choose
$\gamma_{1},\dots,\gamma_{2g}$ to be standard generators of the fundamental
group for the closed surface $S_g$, and $\gamma_{2g+1},\dots,\gamma_{2g+p}$ to
be curves around the $p$ punctures of $S$. Then $\pi_1(\Gamma_\phi)$ has a
presentation of the form:
\begin{equation*}
	 \langle\gamma_1,\dots,\gamma_{2g+p},\tau | \tau\gamma_i\tau^{-1} = \phi(\gamma_i),
		\Pi_{i=1}^g [\gamma_{2i-1},\gamma_{2i}] = \Pi_{s=1}^p \gamma_{2g+s}\rangle.
\end{equation*}
With these generators for $\pi_1(S)$, $\phi^*:H^1(S)\rightarrow H^1(S)$
can be written as a block matrix
\begin{equation*}
	\begin{pmatrix}[\bar\phi^*] & [*]\\ 0 & [P]\end{pmatrix}
\end{equation*}
where $\bar\phi^*:H^1(\bar{S}) \rightarrow H^1(\bar{S})$ is the induced map on
the first cohomology of the closed surface $\bar{S}$ obtained by filling in the
$p$ punctures of $S$, and $P=(p_{ij})$ is a permutation matrix denoting the
permutation of the punctures of $S$ under $\phi$. In particular, $p_{jk_j}=1$
if and only if $\tau\delta_j\tau^{-1}$ is conjugate to $\delta_{k_j}$, and
$p_{jk_j}=0$ otherwise. The matrix $\bar{\phi}^*$ is a symplectic matrix
preserving the intersection form $\omega$ on $\bar{S}$. The eigenvalues of
$P$ are roots of unity, with 1 occurring as an eigenvalue once for each cycle
in the permutation.

We now compute the cohomological dimension of $H^1(\Gamma_\phi;R_{2j})$.
The argument uses similar ideas to \cite[Theorem 4.1]{kozai16} using the
generators $X^{\ell-1}Y^{2j-\ell}$, $\ell=0,..,2j$, of $R_{2j}$ and is
equivalent up to a coordinate change when $j=1$.

\begin{prop}\label{prop:dimension}
	Let $\phi:S \rightarrow S$ be a homeomorphism, with $\lambda^2$ a simple
	eigenvalue of $\phi^*$. Suppose also that $|\lambda| \neq 1$, 
	$\bar\phi^*: H^1(\bar{S}) \rightarrow H^1(\bar{S})$ does not have 1 as an
	eigenvalue, and for each $2 \leq j \leq n-1$, we have that
	$\lambda^{2j}$ is not an eigenvalue of $\phi^*$. Then for each $j$, $1 \leq j \leq n-1$,
	$\dim H^1(\Gamma_\phi; R_{2j}) = k$ where $k$ is the number of components of
	$\partial M_\phi$.
\end{prop}

\begin{proof}
	Let $z \in Z^1(\Gamma_\phi,R_{2j})$.
	Then $z$ is determined by its values on $\gamma_1$, $\dots$, $\gamma_{2g+p}$,
	and $\tau$, subject to the cocycle condition $\eqref{eqn:cocycle}$ imposed by
	the relations in $\Gamma_\phi$. These can be computed via the Fox calculus
	\cite[Chapter 3]{lubotzky85}.
	Differentiating the relations
	\begin{align*}
		\tau \gamma_i \tau^{-1} &= \phi(\gamma_i),
	\end{align*}
	yields
	\begin{align}
		\frac{\partial [ \phi(\gamma_i) \tau \gamma_i^{-1} \tau^{-1}]}{\partial \gamma_i}
			&= \frac{\partial\phi(\gamma_i)}{\partial \gamma_i}
				- \phi(\gamma_i)\tau\gamma_i^{-1}
				= \frac{\partial\phi(\gamma_i)}{\partial \gamma_i} - \tau\label{eqn:derivatives}\\
		\frac{\partial [ \phi(\gamma_i) \tau \gamma_i^{-1} \tau^{-1}]}{\partial \gamma_h}
			&= \frac{\partial\phi(\gamma_i)}{\partial \gamma_h}, i \neq h\notag\\
		\frac{\partial [ \phi(\gamma_i) \tau \gamma_i^{-1} \tau^{-1}]}{\partial \tau}
			&= \phi(\gamma_i) - \phi(\gamma_i)\tau\gamma_i^{-1}\tau^{-1}
				= \phi(\gamma_i) - 1.\notag
	\end{align}
	A cocycle $z$ then must satisfy the set of equations for $1 \leq i \leq 2g+p$ of
	the form
	\begin{equation}
		\sum_{h=1}^{2g+p}\frac{\partial [ \phi(\gamma_i) \tau \gamma_i^{-1} \tau^{-1}]}{
			\partial \gamma_h} \cdot z(\gamma_h)
		+ \frac{\partial [ \phi(\gamma_i) \tau \gamma_i^{-1} \tau^{-1}]}{
			\partial \tau} \cdot z(\tau) = 0.\label{eqn:kernel}
	\end{equation}
	
	With respect to the basis $X^0Y^{2j}$, $X^1Y^{2j-1}$, \dots, $X^{2j}Y^0$
	for $R_{2j}$, the values $z(\gamma_i)$ can be expressed in coordinates
	$(x_{i,\ell})$, where $x_{i,\ell}$ is the coefficient of $X^\ell Y^{2j-\ell}$
	for $z(\gamma_i)$. We similarly express $z(\tau)$ in the coordinates 
	$x_{0,\ell}$, $0\leq \ell \leq 2j$ with $x_{0,\ell}$ being the $X^\ell
	Y^{2j-\ell}$ coefficient of $z(\tau)$. Direct calculation shows that
	\begin{align}
		\rho(\gamma_i) \cdot X^\ell Y^{2j-\ell} &=(X-a_iY)^\ell Y^{2j-\ell}
			\label{eqn:action} \\
			&= \sum_{m=0}^{\ell}(-a_i)^m\binom{\ell}{m}X^{\ell-m}Y^{2j-\ell+m},\notag \\
		\rho(\tau) \cdot X^{\ell}Y^{2j-\ell} &=(\lambda^{-1}X)^\ell
			(\lambda Y)^{2j-\ell}\notag \\
			&= \lambda^{2j-2\ell}X^\ell Y^{2j-\ell}. \notag 
	\end{align}
	
	The set of coboundaries can be computed from Equation \eqref{eqn:coboundary}
	as the set of cocycles $z'$ satisfying,
	\begin{align*}
		z'(\gamma_i) &= \sum_{\ell=0}^{2j} b_{\ell}X^\ell Y^{2j-\ell}
			-b_{\ell}(X-a_iY)^\ell Y^{2j-\ell}\\
			&=\sum_{\ell=0}^{2j} \sum_{m=1}^\ell -b_{\ell}(-a_i)^m \binom{\ell}{m}
			X^{\ell-m}Y^{2j-\ell+m},\\
		z'(\tau) &= \sum_{\ell=0}^{2j}(b_{\ell}-\lambda^{2j-2\ell}b_{\ell})
			X^\ell Y^{2j-\ell},
	\end{align*}
	where $b_0,\dots,b_{2j} \in \mathbb{C}$ parametrize the set
	$B^1(\Gamma_\phi,R_{2j})$ of coboundaries. In particular, adding the
	appropriate coboundary $z'$ to $z$, we can assume $x_{0,\ell}=0$ for
	$\ell \neq j$, so that $z(\tau)$ has the form
	\begin{equation*}
		z(\tau) = x_{0,j}X^jY^j.
	\end{equation*}
	
	Then $z$ is determined by a vector
	\begin{equation*}
		\vec{v}=(x_{1,0},\dots,x_{2g+p,0},\dots,x_{0,j},x_{1,j},\dots,x_{2g+p,j},\dots,x_{1,2j},\dots,x_{2g+p,2j})^T
	\end{equation*}
	in the kernel of a block matrix $A=\begin{pmatrix}A_{\alpha,\beta}\end{pmatrix}$
	where the entries in the $i$-th row of $A_{\alpha,\beta}$ are the coefficients
	of the terms $x_{*,\beta}X^\alpha Y^{2j-\alpha}$ in Equation \eqref{eqn:kernel}.
	Since the image under $\rho$ of any word $w$ in $\{\gamma_i,
	\gamma_i^{-1}\}_{i=1}^{2g+p}$
	has the form
	\begin{equation*}
		\rho(w)  = \begin{pmatrix} 1 & a \\ 0 & 1 \end{pmatrix}
	\end{equation*}
	for some $a \in \mathbb{C}$, then the previous calculations in Equations
	\eqref{eqn:action} imply that $A_{\alpha,\beta}=\mathbf{0}$ for $\beta<\alpha$.
	Moreover, when $\alpha \neq j$, $A_{\alpha,\alpha}$ is a square matrix, and we
	note that the coefficient of $X^\alpha Y^{2j-\alpha}$ in $\rho(\gamma_i) \cdot
	X^\alpha Y^{2j-\alpha}$ is 1, so that in Equation \eqref{eqn:kernel}, the
	coefficient of $x_{h,\alpha}$ in the $X^\alpha Y^{2j-\alpha}$ term is
	the signed number of times that $\gamma_h$ appears in the word $\phi(\gamma_i)$.
	In addition, Equation \eqref{eqn:kernel} will contain a single $-\tau \cdot
	z(\gamma_i)$ term, so that $A_{\alpha,\alpha}=\phi^* - \lambda^{2j-2\alpha} I$
	when $\alpha \neq j$. We also see that
	\begin{equation*}
		A_{j,j} = \begin{pmatrix} \begin{matrix} 0 \\ \vdots \\ 0 \end{matrix} &
			\begin{bmatrix} \\ \phi^* - I \\ \\  \end{bmatrix}
			\end{pmatrix}.
	\end{equation*}
	Since $z(\tau)=x_{0,j}X^jY^j$, direct calculation also shows that for some matrix
	$K$,
	\begin{equation*}
		A_{j-1,j}=\left(\begin{array}{c|c}
			\begin{matrix} -j\lambda^2 a_1  \\ \vdots \\
			-j\lambda^2 a_{2g+p} \end{matrix}  & 
			K \end{array}\right).
	\end{equation*}

	As $\lambda^2$ is a simple eigenvalue, $\bar\phi^*$ is symplectic, and the
	eigenvalues of $P$ are roots of unity,
	$\phi^*-\lambda^2 I$ and $\phi^*-\lambda^{-2} I$ have
	1 dimensional kernel. Furthermore, since 1 is not an eigenvalue of $\bar\phi^*$,
	$\phi^*-I$ has kernel whose dimension is
	equal to the number of disjoint cycles of the permutation of the punctures.
	This is equal to the number of components of $\partial M_\phi$. In addition,
	since $\lambda^{2j-2\alpha}$ is not an eigenvalue of $\phi^*$ for $\alpha \neq
	j-1, j, 1$, the kernel of $A_{\alpha,\alpha}$ is trivial in these cases.
	Hence, the kernel of $A$ has dimension at most $2+k+1$, where 
	\begin{equation*}
		k=\text{\# of components of }\Sigma=\text{\# of components of }\partial M_\phi.
	\end{equation*}
	The additional dimension comes from the possible contribution to the kernel from
	the first column of $A_{(j-1),j}$. Consider the submatrix
	\begin{align*}
		U &=\left( \begin{array}{c|c} A_{j-1,j-1} & A_{j-1,j}\\
		\hline \mathbf{0} & A_{j,j} \end{array} \right)\\
		&= \left(\begin{array}{c|c|c}
			\phi^*-\lambda^2 I & \begin{array}{c}
					-j \lambda^2 a_1 \\ \vdots \\-j \lambda^2 a_{2g+n}\end{array}
				& K \\ \hline
			\mathbf{0} & \begin{array}{c} 0 \\ \vdots \\0 \end{array}
				& \phi^*-I
		\end{array}\right).
	\end{align*}
	If $\text{null}(A) > 2+k$, then we must have that $\text{null}(U) > k+1$.
	
	Since $\lambda^2$ is a simple eigenvalue of $\phi^*$ and
	$(a_1,\dots,a_{2g+p})^T$ is an eigenvector of the $\lambda^2$ eigenspace,
	$(a_1,\dots,a_{2g+p})^T$ is not in the image of $\phi^* - \lambda^2 I$. Hence,
	for any $x=(x_{1,j},\dots,x_{2g+p,j})^T$ in the kernel of $\phi^* - I$, there
	is a unique $x_{0,j}$ such that $Kx-x_{0,j}(a_1,\dots,a_{2g+p})^T$ is in the
	image of $\phi^*-\lambda^2 I$. Therefore, $\text{null}(U) = k+1$.
	
	Hence $\text{null}(A) = 2+k$. However, the solution arising from
	the kernel of $\phi^*-\lambda^2 I$ is the eigenvector
	\begin{equation*}
		(0,\dots, 0, x_{1,j}, \dots, x_{2g+p,j}, 0, \dots, 0)^T
		= (0, \dots, 0, a_1,\dots,a_{2g+p},0,\dots,0)^T
	\end{equation*}
	which	is a coboundary. So we have that $\dim H^1(\Gamma_\phi; R_{2j})
	\leq k+1$. Finally, there is one further redundancy since 
	\begin{equation*}
		\Pi_{i=1}^g [\gamma_{2i-1},\gamma_{2i}]=\Pi_{s=1}^p \gamma_{2g+s}.
	\end{equation*}
	From the $\phi^*-I$ in $A_{j,j}$, we can see that $x_{j,2g+1},\dots,x_{j,2g+p}$
	can be freely chosen as long as $x_{j,2g+s}=x_{j,2g+t}$ whenever
	$\gamma_{2g+s}$ and $\gamma_{2g+t}$ are in the same cycle of $P$.
	Since $|\lambda|\neq 1$, for any eigenvector of $\phi^*$, $a_{2g+1}=\cdots=
	a_{2g+p}=0$, so the $X^jY^j$ coefficient of $z(\Pi_{s=1}^n \gamma_{2g+s})$ can be
	chosen to be any quantity
	\begin{equation}
		x_{j,2g+1} + \dots + x_{j,2g+p}.\label{eqn:j-sum}
	\end{equation}
	The relation $\Pi_{i=1}^g [\gamma_{2i},\gamma_{2i+1}]=\Pi_{s=1}^p \gamma_{2g+s}$
	relates the sum in Equation \eqref{eqn:j-sum} to the
	$X^jY^j$ coefficient of $\Pi_{i=1}^g [\gamma_{2i},\gamma_{2i+1}]$, which has
	no dependence on $x_{j,2g+s}$, for $1\leq s \leq p$.
	This imposes a 1-dimensional relation on the space of cocycles, and we conclude
	that
	\begin{equation*}
		\dim H^1(\Gamma_\phi, R_{2j})= k.
	\end{equation*}
\end{proof}

We now prove Theorem \ref{thm:smooth} and Theorem \ref{thm:pseudoanosov}.

\begin{proof}[Proof of Theorem \ref{thm:smooth}]
	By Lemma \ref{lem:decomposition}, $\mathfrak{sl}(n)$ is the direct sum
	of $R_{2j}$, $j=1,\dots,n-1$. The conditions on the eigenvalues of $\phi^*$ and
	Proposition \ref{prop:dimension} imply that for each $j$,
	$\dim H^1(\Gamma_\phi;R_{2j}) = k$. Hence
	$\dim H^1(\Gamma_\phi,\mathfrak{sl}(n)_{\rho_{\lambda,n}}) = k(n-1)$.
	By Proposition \ref{prop:dim_smooth}, this implies smoothness of
	$R(\Gamma_\phi,\SL(n))$ at $\rho_{\lambda,n}$. Since $\rho_{\lambda,n}$
	is non-abelian, it has trivial infinitesimal centralizer, so $H^0(\Gamma_\phi;R_2) = 0$,
	so that the local dimension is $(n+1+k)(n-1)$.
\end{proof}

We obtain the special case in Theorem \ref{thm:pseudoanosov} when $\lambda^2$
is the dilatation of a pseudo-Anosov
map $\phi$. 

\begin{proof}[Proof of Theorem \ref{thm:pseudoanosov}]
	When the stable and unstable foliations of $\phi$ are orientable,
	it is a well-known fact that the dilatation is a simple eigenvalue and the
	largest eigenvalue of $\phi^*$ (see \cite{fathi79}, \cite{mcmullen03},
	\cite{penner91}). Hence, $\phi$ satisfies the conditions of
	Theorem \ref{thm:smooth}.
	
	From \cite{kozai16}, we know that there are hyperbolic deformations of
	$\rho_\lambda=\rho_{\lambda,2}$, which are irreducible representations
	since they correspond to hyperbolic structures. The composition of these
	deformations with
	the irreducible representation $r_n$ then provides nearby deformations
	of $\rho_{\lambda,n}$	which are also irreducible.
\end{proof}

\section{Description of Deformations}

Recall that the action of $\Gamma_\phi$ on $\mathfrak{sl}(n)$ is given by composing
$\rho_{\lambda,n}$ with the adjoint representation. That is, for $\gamma \in
\Gamma_\phi$ and $c \in \mathfrak{sl}(n)$,
\begin{equation*}
	\gamma \cdot c = \Ad{\rho_{\lambda,n}(\gamma)}(c)=
		\rho_{\lambda,n}(\gamma) ~c~ \rho_{\lambda,n}(\gamma)^{-1}.
\end{equation*}
Let $E_{j}$ denote the $j$-th standard basis vector for $\mathbb{C}^n$. Then every
element of $\mathfrak{sl}(n)$ is a linear combination of the matrices
$E_{j,\ell} = E_j\cdot E_\ell^T$. In order to obtain a useful description of the
action of $\Gamma_\phi$ on $\mathfrak{sl}(n)$, it suffices to compute the action
of $\gamma$ on $E_{j,\ell}$ for a set of generators of $\Gamma_\phi$. By direct
calculation,
\begin{align}
	\gamma_i \cdot E_{j,\ell} &= r_n \left( \begin{bmatrix} 1 & a_i \\ 0 & 1
		\end{bmatrix}\right) E_j\cdot E_\ell^T r_n
		\left( \begin{bmatrix} 1 & a_i \\ 0 & 1 \end{bmatrix}\right)^{-1}
		\label{eqn:slaction} \\
		&= \begin{pmatrix} (-a_i)^{j-1} \binom{j-1}{\ell-1}\\
			(-a_i)^{j-2} \binom{j-1}{\ell-2} \\
			\vdots \\
			(-a_i)^0 \binom{j-1}{0} \\
			0 \\ \vdots \\ 0 \end{pmatrix}
			\begin{pmatrix} 0 \\ \vdots \\ 0 \\ 1 \\ (a_i)^1 \binom{\ell}{1} \\
				(a_i)^2\binom{\ell+1}{2} \\ \vdots \\
				(a_i)^{n-\ell} \binom{n-1}{n-\ell} \end{pmatrix}^T
				\notag \\
	\tau \cdot E_{j,\ell} & = r_n \left( \begin{bmatrix} \lambda & 0 \\
		0 & \lambda^{-1} \end{bmatrix} \right) E_j \cdot E_\ell^T r_n \left(
			\begin{bmatrix} \lambda & 0 \\ 0 & \lambda^{-1} \end{bmatrix} \right)^{-1}\notag \\
		&= \lambda^{2(n-j+1)} E_j \cdot E_\ell^T \lambda^{-2(n-\ell+1)} \notag\\
		&= \lambda^{2(\ell-j)}E_{j,\ell}.\notag 
\end{align}
Notably, the actions of
$\Gamma_\phi$ on the $(j,\ell)$-coordinates of $\mathfrak{sl}(n)$ have no
contributions to all rows $>j$ and all columns $<\ell$. Applying analogous
calculations as in the proof of Proposition \ref{prop:dimension} to the
relations in $\Gamma_\phi$, we find that if $z:\Gamma_\phi \rightarrow
\mathfrak{sl}(n)_{\rho_{\lambda,n}}$ is a cocycle and $z_{j,\ell}(\gamma_i)$ is
the $(j,\ell)$-coordinate of $z(\gamma_i)$, then the vector
\begin{equation*}
	\vec{v}_{n,1} = \begin{pmatrix} z_{n,1}(\gamma_1) \\ \vdots \\
		z_{n,1}(\gamma_{2g+p}) \end{pmatrix} = (z_{n,1}(\gamma_i))
\end{equation*}
is a solution to $(\phi^* - \lambda^{-2(n-1)}I)\vec{v}_{n,1}=\mathbf{0}$. Since
$\lambda^{-2(n-1)}$ is not an eigenvalue of $\phi^*$, it follows that
$\vec{v}_{n,1}=\mathbf{0}$.

Since $\vec{v}_{n,1}=\mathbf{0}$, when the relations in
$\Gamma_\phi$ applied to $z$ are restricted to the $(n-1,1)$-coordinate and the
$(n,2)$-coordinate, we obtain that $\vec{v}_{n-1,1} = (z_{n-1,1}(\gamma_i))$ and
$\vec{v}_{n,2} = (z_{n,2}(\gamma_i))$ are solutions to $(\phi^*-\lambda^{-2(n-2)}I)
\vec{v} =\mathbf{0}$. A straightforward induction combined with
Equations \eqref{eqn:slaction} then shows that $z_{j,\ell}(\gamma_i)=0$
for all $j>\ell+1$  while $\vec{v}_{j,\ell}=(z_{j,\ell}(\gamma_i))$ is a
$\lambda^{-2}$-eigenvector of $\phi^*$ when $j=\ell+1$, i.e. the subdiagonal
entries of $z(\gamma_i) \in \mathfrak{sl}(n)$ are coordinates from eigenvectors
of $\phi^*$, and all other entries below the diagonal are $0$. This provides
$n-1$ generators of cocycles. The others come from the 1-eigenspaces of
$\phi^*$ when applying the cocycle conditions to the diagonal entries of
$z(\gamma_i)$.

We have that $\mathfrak{sl}(n)$ can be associated with the tangent space to
$\SL(n)$ at the identity, and multiplying $z(\gamma_i)$ by
$\rho_{\lambda,n}(\gamma_i)$ gives the derivative at $\rho_{\lambda,n}(\gamma_i)$.
The previous calculations then imply that if $\rho_t:\Gamma_\phi
\rightarrow \SL(n)$ is a path of representations such that $\rho_0 =
\rho_{\lambda,n}$, then the subdiagonal entries
of $\rho_t'(\gamma_i)$ at $t=0$ are equal to the subdiagonal entries of
$z(\gamma_i)$. Hence, for each $j,\ell$, there exists at least one $i$ for which
$z_{j,\ell}(\gamma_i) \neq 0$.

Note that in the case that $\lambda^2$ is the dilatation of a pseudo-Anosov
map $\phi$ as in Theorem \ref{thm:pseudoanosov}, the subdiagonal entries of
the irreducible representations
obtained by deforming $\rho_\lambda$ in $\SL(2)$ and composing with $r_n$ to
obtain a deformation of $\rho_{\lambda,n}$ necessarily
satisfy certain relations. In particular, the first derivatives of the
subdiagonal entries would have to be fixed multiples of entries of the
$\lambda^{-2}$-eigenvector determined by the irreducible representation
$r_n$. As described above, the deformations in
$\SL(n)$ allow the derivatives to be freely chosen multiples of the $n-1$
generators, so there are deformations which are not from deformations of
$\rho_\lambda$ that are composed with $r_n$. Since the set of irreducible
representations is an open subset of the space of $R(\pi_1(M_\phi),\SL(n))$
(see, for example, \cite[Lemma 1.4.2]{culler83} \cite[Proposition 27]{sikora12}),
this also implies there are nearby irreducible representations which are
not from composing deformations of $\rho_\lambda$ with $r_n$.

\section{Example}

 

The genus 2 example $\phi:S_{2,2} \rightarrow S_{2,2}$ from
\cite{kozai16}, obtained from 
taking the left Dehn twists $T_{\beta_1}, T_{\beta_2},T_\gamma$, followed by the
right Dehn twists $T_{\alpha_1}^{-1}, T_{\alpha_2}^{-1}$, satisfies the hypotheses
of Theorem \ref{thm:pseudoanosov}. Each component
of $S_2 \setminus \{\alpha_1,\beta_1,\alpha_2,\beta_2,\gamma\}$ contains one of
the two punctures.
The map on cohomology $\bar{\phi}^*$ has two simple eigenvalues
$\lambda_1^2=\frac{5+\sqrt{21}}{2}$ and $\lambda_{2}^2 = \frac{3+\sqrt{5}}{2}$, along
with their reciprocals $\lambda_1^{-2}$ and $\lambda_2^{-2}$. The reducible
representations $\rho_{\lambda_i,n}$ are smooth points of 
$R(\Gamma_\phi,\SL(n))$, each on a component of dimension $(n+3)(n-1)$.
There is a two-dimensional family of irreducible representations in
$X(\Gamma_\phi,\SL(n))$, which is the image of a two-dimensional family
of irreducible representations in $X(\Gamma_\phi,SL(2))$ under $r_n$,
limiting to $\rho_{\lambda_1,n}$.

 \begin{figure}
	\begin{center}
		\includegraphics{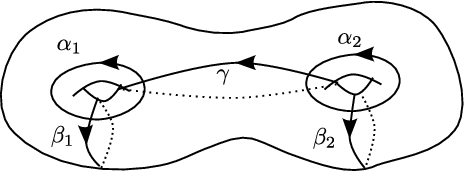}
		\caption{The curves $\alpha_1,\alpha_2,\beta_1,\beta_2$ which form the
			basis for $H_1(S)$, and $\gamma$.}
		\label{fig:genus2}
	\end{center}
\end{figure}

\bibliographystyle{amsplain}
\bibliography{sources1}

\end{document}